\documentclass[11pt]{article}

\usepackage{amsmath,amssymb,amsthm,mathrsfs}
\usepackage[a4paper,margin=1in]{geometry}

\newtheorem{theorem}{Theorem}[section]
\newtheorem{proposition}[theorem]{Proposition}
\newtheorem{corollary}[theorem]{Corollary}
\newtheorem{lemma}[theorem]{Lemma}
\newtheorem{example}[theorem]{Example}
\newtheorem{remark}[theorem]{Remark}
\newtheorem{definition}[theorem]{Definition}

\newcommand{\R}{\mathbb R}

\newcommand{\email}[1]{\texttt{#1}}

\usepackage{tikz}
\usetikzlibrary{arrows,calc,decorations.pathreplacing}
\usepackage{float}

\begin{document}
	
	\title{A weighted Birkhoff orthogonal James-type constant}
	\author{Junxiang Qi$^{1}$, Zhouping Yin$^{1}$, Qi Liu$^{1*}$, Yongjin Li$^{2}$}
	\date{}
	
	\maketitle
	

	{\small 1 School of Mathematics and Statistics
		Anqing Normal University
		Anqing--246133, P. R. China
		
		\email{y25060009@stu.aqnu.edu.cn},\quad
		\email{yzp@aqnu.edu.cn},\quad
		\email{liuq67@aqnu.edu.cn}}
	
	\vspace{0.3em}
	
	{\small 2 Department of Mathematics
		Sun Yat-sen University
		Guangzhou--510275, P. R. China
		
		\email{stslyj@mail.sysu.edu.cn}}

\begin{abstract}
Let $X$ be a real Banach space and $\lambda \in[0,1]$. Motivated by orthogonal versions of the James constant, we introduce the weighted Birkhoff orthogonal James-type constant $$J_\lambda^{\perp}(X)=\sup \left\{\min \{\|\lambda x+(1-\lambda) y\|,\|\lambda x-(1-\lambda) y\|\}: x, y \in S_X, x \perp_B y\right\},$$ where \(\lambda\in[0,1]\) and $x \perp_B y$ stands for Birkhoff orthogonality. We establish its basic bounds, stability properties, and reduction principles, and clarify its relations with the orthogonal James constant $J_{\perp}(X)$. The 2 -Lipschitz continuity of $J_\lambda^{\perp}(X)$ with respect to $\lambda$ is proved. New characterizations of uniformly nonsquare spaces are obtained; in particular, $J_\lambda^{\perp}(X)=1$ for some $\lambda \in(0,1)$ if and only if $X$ is not uniformly nonsquare. We also discuss connections with strict convexity, uniform convexity, modulus of smoothness, and the von Neumann-Jordan constant. 

\end{abstract}

\section{ Introduction
	 and preliminaries}
	 
The geometry of Banach spaces has long been a central topic in functional analysis, focusing on the interplay between norm structure, geometric properties, and topological behavior. A real Banach space 
 \(X\)
is a complete normed vector space, where the norm encodes key geometric information such as convexity, smoothness, and reflexivity. These properties determine the behavior of operators, the existence of fixed points, and the structure of unit balls and spheres, making Banach space geometry indispensable in nonlinear analysis, approximation theory, and operator theory, can refer to \cite{1,2,3,4}.

A powerful way to quantify geometric properties is through geometric constants, which assign a numerical value to a geometric feature and enable precise comparisons between different spaces. Representative examples include the James constant, the von Neumann-Jordan constant, the modulus of convexity and the smoothness constant. The James constant, introduced by Gao and Lau, measures the deviation of a space from Euclidean geometry and plays a critical role in characterizing uniform nonsquareness. For more papers on geometric constants,refer to \cite{5,6,7,9,10}.

Birkhoff-James orthogonality is fundamental in Banach space geometry\cite{8}, as it reflects the directional symmetry of the norm and underlies many characterizations of convexity, smoothness, and reflexivity.

Let \(X\) be a real Banach space and let \(S_X=\{x\in X:\|x\|=1\}\).
Recall that \(x\) is said to be Birkhoff--James orthogonal to \(y\), written
\(x\perp_B y\), if
\[
\|x\|\leq \|x+ty\| \qquad \text{for every } t\in R.
\]

Recently, various geometric constants related to Birkhoff orthogonality have been defined(see \cite{19,20,21,22}):
 $$
 D^{\prime}(X)=\sup \left\{\|x+y\|-\|x-y\|: x, y \in S_X, x \perp_B y\right\} .
 $$

$$
	B R(X)=\sup _{\alpha>0}\left\{\|\alpha x+y\|-\|\alpha x-y\|: x, y \in S_X, x \perp_B y\right\}.
$$

 $$
A_2(X, B)=\sup \left\{\frac{\|x+y\|+\|x-y\|}{2}: x, y \in S_X, x \perp_B y\right\} .
$$

$$
\begin{aligned}
	C_B(X) =\sup \left\{\frac{\|x+y\|^2-\|x-y\|^2}{4}: x, y \in S_X, x \perp_B y\right\}.
\end{aligned}
$$
We note that these suprema are considered only in the unit sphere $S_X$.

Recall that the Banach space $X$ was called uniformly non-square \cite{11} if there exists a $\delta \in(0,1)$ such that for any $x, y \in S_X$ either $\frac{\|x+y\|}{2} \leqslant 1-\delta$ or $\frac{\|x-y\|}{2} \leqslant 1-\delta$.

The modulus of smoothness \cite{17} of a space $X$, denoted by the function $\rho(t)$, is defined as follows:

$$
\rho(t)=\sup \left\{\frac{\left\|x+t y\right\|+\left\|x-t y\right\|}{2}-1: x, y \in S_X\right\} .
$$

The James constant $J(X)$ of a Banach space $X$ was introduced by Gao and Lau \cite{18} as follows:

$$J(X)=\sup \left\{\min \{\|x+y\|,\|x-y\|\}: x, y \in S_X\right\} .$$

 Restricting the James constant to Birkhoff-James orthogonal pairs yields the orthogonal James constant\cite{10}

$$
J_{\perp}(X)=\sup \left\{\min \{\|x+y\|,\|x-y\|\}: x, y \in S_X, x \perp_B y\right\},
$$
which precisely characterizes uniform nonsquareness: $X$ is uniformly nonsquare if and only if
$
J_{\perp}(X)<2 .
$

The constant $J(t, X)$

$$
J(t, X)=\sup \left\{\min \{\|x+t y\|,\|x-t y\|\}: x, y \in S_X\right\}
$$
was studied by He and Cui \cite{13} as a generalization of the constant $J(X)$.

Generalized forms of the James constant and local James constant were proposed in literature \cite{14}.
For any $\lambda \in(0,1)$, the generalized James constant $J(\lambda, X)$ is defined by

$$
J(\lambda, X)=\sup \left\{\min \{\|\lambda x+(1-\lambda) y\|,\|\lambda x-(1-\lambda) y\|\}: x, y \in S_{\mathbb{X}}\right\} .
$$

The paper is structured as follows: we first review basic concepts and classical geometric constants in Banach space geometry, including Birkhoff-James orthogonality, the orthogonal James constant, uniform nonsquareness, modulus of smoothness, and the von Neumann-Jordan constant. In Section 2, we formally define the weighted Birkhoff-James orthogonal James-type constant $J_\lambda^{\perp}(X)$ and establish its elementary properties, such as basic bounds, subspace reduction principles, attainment in finite-dimensional spaces, and exact values for Hilbert spaces and $\ell_p^n$ spaces. In Section 3, we derive two-sided inequalities relating $J_\lambda^{\perp}(X)$ to the orthogonal James constant, prove its 2-Lipschitz continuity with respect to $\lambda$, and discuss applications in characterizing uniformly nonsquare spaces, strict convexity, uniform convexity, modulus of smoothness, and the von Neumann-Jordan constant, along with key corollaries for geometric properties and fixed point theory.

\section{The new constant \(J_\perp(X)\) and elementary properties}

Throughout the paper we assume \(\dim X\geq 2\), so that non-trivial
Birkhoff--James orthogonal pairs exist.

We now define the weighted version.
Despite its success, the unweighted orthogonal James constant lacks flexibility in describing intermediate geometric regimes, where the balance between orthogonal vectors varies continuously. To address this limitation, we introduce a weighted Birkhoff-James orthogonal James-type constant $J_\lambda^{\perp}(X)$ for $\lambda \in[0,1]$, it is defined as follows.

\begin{definition}
	Let \(X\) be a real Banach space and let \(\lambda\in[0,1]\). Define
	\[
	J_\lambda^\perp(X)
	=
	\sup\left\{
	\min\left\{
	\|\lambda x+(1-\lambda)y\|,
	\|\lambda x-(1-\lambda)y\|
	\right\}
	:
	x,y\in S_X,\ x\perp_B y
	\right\}.
	\]
\end{definition}

For $x, y \in S_X$ and $\lambda \in[0,1]$, put

$$
m_\lambda(x, y)=\min \{\|\lambda x+(1-\lambda) y\|,\|\lambda x-(1-\lambda) y\|\} .
$$

Thus

$$
J_\lambda^{\perp}(X)=\sup \left\{m_\lambda(x, y): x, y \in S_X, x \perp_B y\right\} .
$$

\begin{figure}[H]
\centering
\begin{tikzpicture}[>=stealth, scale=1.3]
	
	\coordinate (O) at (0,0);
	
	\coordinate (x) at (1.2,1.8);
	\coordinate (lambda_x) at (0.6,0.9); 
	\coordinate (p_plus) at (3,2.5);     
	\coordinate (p_minus) at (3,-2.5);  
	
	\draw[blue,->,thick] (O) -- (x) node[above left,blue] {$x$};
	
	\draw[black,thick] (O) -- (lambda_x);
	\node[black,above left] at (0.7,0.85) {$\lambda x$};
	
	\draw[blue,dashed] (x) -- (p_plus) node[midway,above left,blue] {$+(1-\lambda)y$};
	
	\draw[black,dashed] (lambda_x) -- (p_minus) node[midway,below,right,black] {$-(1-\lambda)y$};
	
	\draw[blue,->,thick] (O) -- (p_plus) node[midway,right,blue] {$\lambda x + (1-\lambda)y$};
	\fill[blue] (p_plus) circle (2pt);
	\node[blue,above] at (p_plus) {$p_+ = \lambda x + (1-\lambda)y$};
	
	\draw[red,->,thick] (O) -- (p_minus) node[midway,below left,red] {$\lambda x - (1-\lambda)y$};
	\fill[red] (p_minus) circle (2pt);
	\node[red,below] at (p_minus) {$p_- = \lambda x - (1-\lambda)y$};
	
	\node[left] at (O) {$O$};
	
	\draw[decorate,decoration={brace,amplitude=10pt},thick]
	($(p_plus)+(1.1,0.2)$) -- ($(p_minus)+(1.1,-0.2)$)
	node[midway,right,align=left] {
		$\min\Big\{\|\lambda x + (1-\lambda)y\|,$ \\
		$\quad\|\lambda x - (1-\lambda)y\|\Big\}$
	};

\end{tikzpicture}
\caption{Geometric construction of weighted orthogonal geometric constant $J_\lambda^\perp(X)$}
\label{fig:jlambda}
\end{figure}
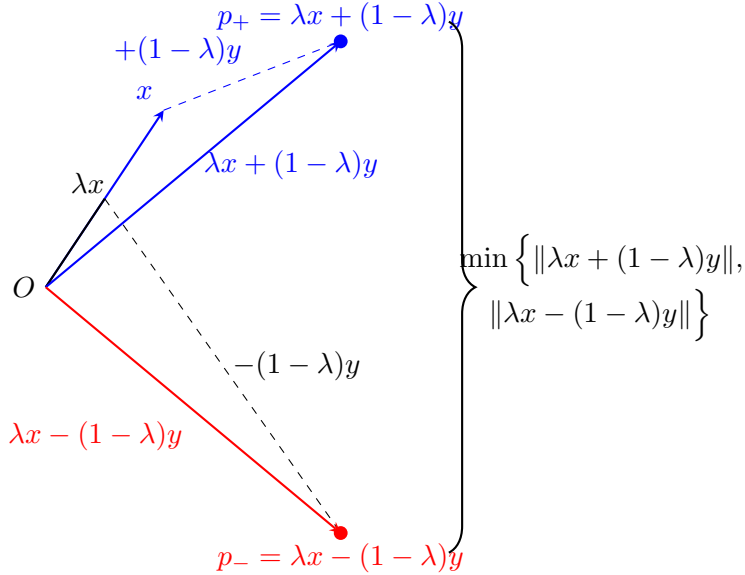

This figure \ref{fig:jlambda} provides an intuitive illustration: given a unit vector $x$, vary the unit vector $y$ under a certain orthogonality condition, and compute the norms of the two combinations $\lambda x \pm(1-\lambda) y$. The geometric constant $J_\lambda^{\perp}(X)$ is defined by comparing their minimum possible maximum (or maximum possible minimum), which characterizes the deviation of orthogonality and inner product properties in a normed space.

\begin{remark}
 In particular, when $\lambda=1 / 2$,
 $
 J_{1 / 2}^{\perp}(X)=\frac{1}{2} J_{\perp}(X)
 $ .
 
\end{remark}
\begin{remark}
		For every real Banach space \(X\).
		
		If \(\lambda=0\), then
		\[
		\|\lambda x+(1-\lambda)y\|=\|y\|=1,
		\qquad
		\|\lambda x-(1-\lambda)y\|=\|-y\|=1.
		\]
		Hence the minimum equals \(1\), and so \(J_0^\perp(X)=1\).
		
		If \(\lambda=1\), then \(J_1^\perp(X)=1\).
\end{remark}

\begin{proposition}
For every real Banach space \(X\) and every
\(\lambda\in[0,1]\),
\[
0\leq J_\lambda^\perp(X)\leq 1.
\]
Moreover, if \(0<\lambda<1\), then
\[
J_\lambda^\perp(X)\geq \max\{\lambda,|2\lambda-1|\}.
\]
\end{proposition}

\begin{proof}
For \(x,y\in S_X\),
\[
\|\lambda x\pm(1-\lambda)y\|
\leq \lambda\|x\|+(1-\lambda)\|y\|=1.
\]
Hence the minimum is at most \(1\), and so \(J_\lambda^\perp(X)\leq 1\).

If \(x\perp_B y\), then for \(\lambda>0\),
\[
\|\lambda x\pm(1-\lambda)y\|
=
\lambda\left\|x\pm \frac{1-\lambda}{\lambda}y\right\|
\geq \lambda\|x\|
=
\lambda.
\]
Also, by the reverse triangle inequality,
\[
\|\lambda x\pm(1-\lambda)y\|
\geq
\big|\lambda-(1-\lambda)\big|
=
|2\lambda-1|.
\]
Thus for every Birkhoff--James orthogonal pair,
\[
\min\{
\|\lambda x+(1-\lambda)y\|,
\|\lambda x-(1-\lambda)y\|
\}
\geq
\max\{\lambda,|2\lambda-1|\}.
\]
Taking the supremum gives the claim.
\end{proof}

\begin{proposition}
	For every real Banach space \(X\),
	\[
	J_\lambda^\perp(X)
	=
	\sup\left\{
	J_\lambda^\perp(E): E\subset X,\ \dim E=2
	\right\}.
	\]
\end{proposition}

\begin{proof}
	First, by subspace monotonicity, for every two-dimensional subspace \(E\),
	\[
	J_\lambda^\perp(E)\leq J_\lambda^\perp(X).
	\]
	Hence
	\[
	\sup_{\dim E=2}J_\lambda^\perp(E)\leq J_\lambda^\perp(X).
	\]
	
	Conversely, let \(x,y\in S_X\) with \(x\perp_B y\), take
	\(
	E=\operatorname{span}\{x,y\}.
	\)
	
	Then \(E\) is at most two-dimensional, if \(x,y\) are linearly independent,
	then \(\dim E=2\). The Birkhoff-James condition
	
	$$
	\|x\| \leq\|x+t y\| \quad(t \in \mathbb{R})
	$$
	
	depends only on the norms of vectors in the two-dimensional subspace $\operatorname{span}\{x, y\}$, so it also holds in $E$. Hence \(x\perp_B y\) in \(E\).
	
	Therefore
	\[
	\min\{
	\|\lambda x+(1-\lambda)y\|,
	\|\lambda x-(1-\lambda)y\|
	\}
	\leq
	J_\lambda^\perp(E).
	\]
	Taking the supremum over all admissible pairs in \(X\) gives
	\[
	J_\lambda^\perp(X)\leq
	\sup_{\dim E=2}J_\lambda^\perp(E).
	\]
	Thus equality holds.
\end{proof}

\begin{proposition}
	If \(X\) is finite-dimensional, then for every \(\lambda\in[0,1]\) there
	exist \(x_0,y_0\in S_X\) with \(x_0\perp_B y_0\) such that
	\[
	J_\lambda^\perp(X)
	=
	\min\left\{
	\|\lambda x_0+(1-\lambda)y_0\|,
	\|\lambda x_0-(1-\lambda)y_0\|
	\right\}.
	\]
\end{proposition}

\begin{proof}
	Since \(X\) is finite-dimensional, \(S_X\times S_X\) is compact.
	
	Define
	\[
	\mathcal O
	=
	\{(x,y)\in S_X\times S_X:x\perp_B y\}.
	\]
	We first show that \(\mathcal O\) is closed. Suppose
	\(
	(x_n,y_n)\in\mathcal O,\) and
	\(x_n\to x, y_n\to y.\)
	
	For every \(t\in\mathbb R\),
	\[
	\|x_n\|\leq\|x_n+ty_n\|.
	\]
	Passing to the limit gives
	\[
	\|x\|\leq\|x+ty\|.
	\]
	Therefore \(x\perp_B y\), so \((x,y)\in\mathcal O\).
	Thus \(\mathcal O\) is closed, and hence compact.
	
	Now the function
	\[
	(x,y)\mapsto
	\min\left\{
	\|\lambda x+(1-\lambda)y\|,
	\|\lambda x-(1-\lambda)y\|
	\right\}
	\]
	is continuous on \(\mathcal O\). Therefore it attains its maximum on
	\(\mathcal O\). This maximum is precisely \(J_\lambda^\perp(X)\).
\end{proof}

\begin{example} If $X$ contains an isometric copy of $\ell_{1}^{2}$, then for every $\lambda \in[0,1]$,
	
	$$
	J_{\lambda}^{\perp}(X)=1 .
	$$
\end{example}

\begin{proof} Let $E \subset X$ be an isometric copy of $\ell_{1}^{2}$. By monotonicity with respect to subspaces,
	
	$$
	J_{\lambda}^{\perp}(X) \geq J_{\lambda}^{\perp}(E) .
	$$
	
	In $\ell_{1}^{2}$, take
	
	$$
	x=(1,0), \quad y=(0,1) .
	$$
	
	Then $x \perp_{B} y$, and
	
	$$
	\|\lambda x \pm(1-\lambda) y\|_{1}=\lambda+(1-\lambda)=1
	$$
	
	Thus
	$
	J_{\lambda}^{\perp}(E)=1 .
	$
	
	Since always $J_{\lambda}^{\perp}(X) \leq 1$, we conclude that
	$
	J_{\lambda}^{\perp}(X)=1.
	$
\end{proof}

\begin{example} Let $1 \leq p<\infty$ and $n \geq 2$. Then
	
	$$
	J_{\lambda}^{\perp}\left(\ell_{p}^{n}\right) \geq\left(\lambda^{p}+(1-\lambda)^{p}\right)^{1 / p}.
	$$
\end{example}
\begin{proof} Let
	
	$$
	e_{1}=(1,0, \ldots, 0), \quad e_{2}=(0,1,0, \ldots, 0) .
	$$
	
	Then $e_{1}, e_{2} \in S_{\ell_{p}^{n}}$. Moreover,
	
	$$
	e_{1} \perp_{B} e_{2},
	$$
	
	because for every $t \in \mathbb{R}$,
	
	$$
	\left\|e_{1}+t e_{2}\right\|_{p}=\left(1+|t|^{p}\right)^{1 / p} \geq 1=\left\|e_{1}\right\|_{p} .
	$$
	
	Hence
	
	$$
	\left\|\lambda e_{1} \pm(1-\lambda) e_{2}\right\|_{p}=\left(\lambda^{p}+(1-\lambda)^{p}\right)^{1 / p} .
	$$
	
	Therefore
	
	$$
	J_{\lambda}^{\perp}\left(\ell_{p}^{n}\right) \geq\left(\lambda^{p}+(1-\lambda)^{p}\right)^{1 / p}.
	$$
\end{proof}

\begin{lemma}\cite{16}\label{l11} A normed linear space $X$, whose dimension is at least three, is an inner product space if and only if Birkhoff orthogonality is symmetric in $X$.
\end{lemma}

Based on the properties of Birkhoff orthogonality, we can obtain the following lemma.

\begin{lemma}\label{p2} Suppose that

$$
u, v \in S_X, \quad u \perp_B v \quad \Longrightarrow \quad v \perp_B u .
$$

Then

$$
x, y \in X, \quad x \perp_B y \quad \Longrightarrow \quad y \perp_B x .
$$
\end{lemma}

The characterization of inner product spaces is an important topic in Banach space theory, and the relevant results are presented below.
\begin{theorem}\label{t11}
	Let $X$ be a real normed space with $\dim X\ge3$. Then the following assertions
	are equivalent.
	
	\begin{enumerate}
		\item $X$ is an inner product space.
		
		\item For every $x,y\in S_X$ with $x\perp_B y$ and every $0<\lambda<1$,
		\[
	m_\lambda(x, y)
		=
		\sqrt{\lambda^2+(1-\lambda)^2}.
		\]
		
		\item For every $x,y\in S_X$ with $x\perp_B y$ and every $0<\lambda<1$,
		\[
	m_\lambda(x, y)
		\ge
		\sqrt{\lambda^2+(1-\lambda)^2}.
		\]
		
		\item For every $x,y\in S_X$ with $x\perp_B y$ and every $t>0$,
		\[
		\min\{\|x+t y\|,\|x-t y\|\}
		\ge
		\sqrt{1+t^2}.
		\]
	\end{enumerate}
\end{theorem}

\begin{proof}
	\textup{(1) $\Rightarrow$ (2).}
	Assume $X$ is an inner product space. Then $x\perp_B y$ if and only if
	$\langle x,y\rangle=0$. Therefore, for $x,y\in S_X$ with $x\perp_B y$,
	\[
	\|\lambda x\pm(1-\lambda)y\|^2
	=
	\lambda^2\|x\|^2+(1-\lambda)^2\|y\|^2
	=
	\lambda^2+(1-\lambda)^2.
	\]
	Taking the minimum gives
	\[
m_\lambda(x, y)
	=
	\sqrt{\lambda^2+(1-\lambda)^2}.
	\]
	
	\textup{(2) $\Rightarrow$ (3).}
	This is immediate.
	
	\textup{(3) $\Leftrightarrow$ (4).}
	Let
	\[
	t=\frac{1-\lambda}{\lambda}.
	\]
	Then
	\[
	\lambda=\frac1{1+t},
	\qquad
	1-\lambda=\frac{t}{1+t}.
	\]
	Hence
	\[
	m_\lambda(x, y)
	=
	\frac1{1+t}
	\min\{\|x+t y\|,\|x-t y\|\}.
	\]
	Also,
	\[
	\sqrt{\lambda^2+(1-\lambda)^2}
	=
	\frac{\sqrt{1+t^2}}{1+t}.
	\]
	Thus the inequality in \textup{(3)} is exactly the inequality in
	\textup{(4)}. Since $0<\lambda<1$ corresponds bijectively to $t>0$, the two
	conditions are equivalent.
	
	\textup{(4) $\Rightarrow$ (1).}
	Take $x,y\in S_X$ with $x\perp_B y$. We need to show that
	\[
	y\perp_B x,
	\]
	that is,
	\[
	\|y+s x\|\ge1
	\qquad(s\in\R).
	\]
	For $s=0$ this is trivial. Let $s\ne0$. Write
	\[
	\|y+s x\|
	=
	|s|\left\|x+\frac{1}{s}y\right\|.
	\]
	Put
	\(
	t=\frac1{|s|}>0.
	\)
	By \textup{(4)},
	\[
	\min\{\|x+t y\|,\|x-t y\|\}
	\ge
	\sqrt{1+t^2}.
	\]
	Therefore both $\|x+t y\|$ and $\|x-t y\|$ are at least $\sqrt{1+t^2}$. We get
	\[
	\left\|x+\frac{1}{s}y\right\|
	\ge
	\sqrt{1+\frac1{s^2}}.
	\]
	Consequently,
	\[
	\|y+s x\|
	=
	|s|\left\|x+\frac{1}{s}y\right\|
	\ge
	|s|\sqrt{1+\frac1{s^2}}
	=
	\sqrt{1+s^2}
	\ge1,
	\]
so $y\perp_B x$. 	 By  Lemma \ref{l11} and  Lemma \ref{p2}, we can deduce that $X$ is an inner product space, as desired.
\end{proof}

\begin{proposition}
	If $X$ is a Hilbert space with $dim H\geq2$, then
	\[
	J_\lambda^\perp(H)
	=
	\sqrt{\lambda^2+(1-\lambda)^2}
	\]
	for every $0<\lambda<1$.
\end{proposition}

\begin{proof}
	In a Hilbert space, Birkhoff orthogonality coincides with inner
	product orthogonality. Thus \(x\perp_B y\) means
	\(
	\langle x,y\rangle=0.
	\)
	
	For \(x,y\in S_H\) with \(x\perp_B y\), we have
	\[
	\|\lambda x+(1-\lambda)y\|^2
	=
	\lambda^2\|x\|^2+(1-\lambda)^2\|y\|^2
	=
	\lambda^2+(1-\lambda)^2.
	\]
	Similarly,
	\[
	\|\lambda x-(1-\lambda)y\|^2
	=
	\lambda^2+(1-\lambda)^2.
	\]
	Hence
	\[
	\min\{
	\|\lambda x+(1-\lambda)y\|,
	\|\lambda x-(1-\lambda)y\|
	\}
	=
	\sqrt{\lambda^2+(1-\lambda)^2}.
	\]
	Taking the supremum gives
	\[
	J_\lambda^\perp(H)
	=
	\sqrt{\lambda^2+(1-\lambda)^2}.
	\]

\end{proof}

\section{Several inequalities and some applications}

To establish the intrinsic relationship between the new constant and the classical orthogonal James constant, we derive the following two-sided inequality.
\begin{proposition}\label{p1}
For every real Banach space \(X\) and every \(\lambda\in[0,1]\),
\[
\frac12J_\perp(X)-2\left|\lambda-\frac12\right|\leq J_\lambda^\perp(X)
\leq
|2\lambda-1|+\min\{\lambda,1-\lambda\}J_\perp(X).
\]
\end{proposition}
\begin{proof}
	Let \(x,y\in S_X\) satisfy \(x\perp_B y\).
	
	First assume \(0\leq\lambda\leq 1/2\). Then \(1-\lambda\geq \lambda\), and
	\[
	\lambda x+(1-\lambda)y
	=
	\lambda(x+y)+(1-2\lambda)y.
	\]
	It follows that
	\[
	\|\lambda x+(1-\lambda)y\|
	\leq
	\lambda\|x+y\|+(1-2\lambda).
	\]
	Similarly,
	\[
	\lambda x-(1-\lambda)y
	=
	\lambda(x-y)-(1-2\lambda)y,
	\]
	which implies
	\[
	\|\lambda x-(1-\lambda)y\|
	\leq
	\lambda\|x-y\|+(1-2\lambda).
	\]
	Hence,
	\[
	\min\{
	\|\lambda x+(1-\lambda)y\|,
	\|\lambda x-(1-\lambda)y\|
	\}
	\leq
	1-2\lambda+\lambda\min\{\|x+y\|,\|x-y\|\}.
	\]
	Taking the supremum yields
	\[
	J_\lambda^\perp(X)
	\leq
	1-2\lambda+\lambda J_\perp(X).
	\]
	
	Next, assume \(1/2\leq\lambda\leq 1\). Then \(\lambda\geq1-\lambda\), and
	\[
	\lambda x+(1-\lambda)y
	=
	(1-\lambda)(x+y)+(2\lambda-1)x.
	\]
	Likewise,
	\[
	\lambda x-(1-\lambda)y
	=
	(1-\lambda)(x-y)+(2\lambda-1)x.
	\]
	Consequently,
	\[
	\min\{
	\|\lambda x+(1-\lambda)y\|,
	\|\lambda x-(1-\lambda)y\|
	\}
	\leq
	2\lambda-1+(1-\lambda)\min\{\|x+y\|,\|x-y\|\}.
	\]
	Taking the supremum gives
	\[
	J_\lambda^\perp(X)
	\leq
	2\lambda-1+(1-\lambda)J_\perp(X).
	\]
	Combining these two cases establishes the desired upper bound.
	
	For the reverse inequality, let \(x,y\in S_X\) with \(x\perp_B y\). We rewrite
	\[
	\lambda x+(1-\lambda)y
	=
	\frac12(x+y)+\left(\lambda-\frac12\right)(x-y).
	\]
	Accordingly,
	\[
	\|\lambda x+(1-\lambda)y\|
	\geq
	\frac12\|x+y\|
	-
	\left|\lambda-\frac12\right|\|x-y\|.
	\]
	Since \(\|x-y\|\leq2\), we obtain
	\[
	\|\lambda x+(1-\lambda)y\|
	\geq
	\frac12\|x+y\|-2\left|\lambda-\frac12\right|.
	\]
	Similarly,
	\[
	\lambda x-(1-\lambda)y
	=
	\frac12(x-y)+\left(\lambda-\frac12\right)(x+y),
	\]
	which yields
	\[
	\|\lambda x-(1-\lambda)y\|
	\geq
	\frac12\|x-y\|-2\left|\lambda-\frac12\right|.
	\]
	Thus,
	\[
	\min\{
	\|\lambda x+(1-\lambda)y\|,
	\|\lambda x-(1-\lambda)y\|
	\}
	\geq
	\frac12\min\{\|x+y\|,\|x-y\|\}
	-
	2\left|\lambda-\frac12\right|.
	\]
	Taking the supremum over all Birkhoff--James orthogonal pairs, we conclude
	\[
	J_\lambda^\perp(X)
	\geq
	\frac12 J_\perp(X)-2\left|\lambda-\frac12\right|.
	\]
\end{proof}

\begin{corollary}
For every real Banach space \(X\) and every \(\lambda\in[0,1]\),
\[
J(X)-1-2\left|\lambda-\frac12\right|\leq J_\lambda^\perp(X)
\leq
|2\lambda-1|+\min\{\lambda,1-\lambda\}J(X).
\]

\end{corollary}
\begin{proof}
	The set over which the supremum for 
\(J_\perp(X)\)
	is defined is a subset of that for \(J(X)\)
	, so we immediately obtain
\[
J_\perp(X)\leq J(X).
\]
The desired conclusion then follows from the Proposition \ref{p1}.

Furthermore, it is known from \cite{10} that  \(
J_{\perp}(X)\geq 2J(X)-2.
\) Combining this inequality yields
\[
J_\lambda^\perp(X)\geq \frac12J_\perp(X)-2\left|\lambda-\frac12\right|\geq J(X)-1-2\left|\lambda-\frac12\right|. \]
\end{proof}

\begin{corollary}
	If
	\(
	J_\perp(X)\leq 2-\delta\)
	for some \(\delta>0\), then for every \(0<\lambda<1\),
	\[
	J_\lambda^\perp(X)
	\leq
	1-\delta\min\{\lambda,1-\lambda\}.
	\]
\end{corollary}

\begin{proof}
	The Proposition \ref{p1}, gives
	\[
	1-J_\lambda^\perp(X)
	\geq
	\min\{\lambda,1-\lambda\}\bigl(2-J_\perp(X)\bigr).
	\]
	If \(J_\perp(X)\leq2-\delta\), then
	\[
	2-J_\perp(X)\geq\delta.
	\]
	Therefore
	\[
	1-J_\lambda^\perp(X)
	\geq
	\delta\min\{\lambda,1-\lambda\},
	\]
	which is equivalent to
	\[
	J_\lambda^\perp(X)
	\leq
	1-\delta\min\{\lambda,1-\lambda\}.
	\]
\end{proof}

The parameter-dependent geometric constant $J_\lambda^{\perp}(X)$ possesses stable continuous behavior with respect to $\lambda$. Specifically, it satisfies a 2 -Lipschitz continuity condition as follows.
It shows that the dependence of the newly defined orthogonality-type geometric constant on the parameter \(\lambda\) is stable and controllable.
\begin{proposition}
For every real Banach space \(X\), the function
\[
\lambda\mapsto J_\lambda^\perp(X)\in[0,1]
\]
is \(2\)-Lipschitz and the map
\(
\lambda\mapsto J_\lambda^\perp(X)
\)
is continuous on \([0,1]\). That is,
\[
|J_\lambda^\perp(X)-J_\mu^\perp(X)|
\leq
2|\lambda-\mu|
\]
for all \(\lambda,\mu\in[0,1]\).
\end{proposition}

\begin{proof}
Fix \(x,y\in S_X\) with \(x\perp_B y\). Then
\[
\begin{aligned}
&\left|
\|\lambda x+(1-\lambda)y\|
-
\|\mu x+(1-\mu)y\|
\right| \\
&\leq
\|(\lambda-\mu)x+((1-\lambda)-(1-\mu))y\| \\
&=
|\lambda-\mu|\|x-y\|
\leq
2|\lambda-\mu|.
\end{aligned}
\]
Similarly,
\[
\left|
\|\lambda x-(1-\lambda)y\|
-
\|\mu x-(1-\mu)y\|
\right|
\leq
2|\lambda-\mu|.
\]
Taking minima preserves the same Lipschitz bound. Therefore, for every
orthogonal pair,
\[
\left|
m_\lambda(x,y)-m_\mu(x,y)
\right|
\leq
2|\lambda-\mu|,
\]
where
\[
m_\lambda(x,y)
=
\min\{
\|\lambda x+(1-\lambda)y\|,
\|\lambda x-(1-\lambda)y\|
\}.
\]
Taking suprema gives
\[
|J_\lambda^\perp(X)-J_\mu^\perp(X)|
\leq
2|\lambda-\mu|.
\]
\end{proof}

Recall that a Banach space \(X\) is uniformly nonsquare if there exists
\(\varepsilon>0\) such that for every \(x,y\in S_X\),
\[
\min\{\|x+y\|,\|x-y\|\}\leq 2-\varepsilon.
\]
Equivalently\cite{12},
\(J(X)<2.\)

For the orthogonal James constant, one has the corresponding characterization\cite{10}
\[
X \text{ is not uniformly nonsquare}
\quad\Longleftrightarrow\quad
J_\perp(X)=2.
\]

Below we generalize the above identity to weighted spaces.
\begin{theorem}\label{t1}
Let \(0<\lambda<1\). Then
\[
J_\lambda^\perp(X)=1
\quad\Longleftrightarrow\quad
J_\perp(X)=2.
\]
Consequently,
\[
J_\lambda^\perp(X)=1
\quad\Longleftrightarrow\quad
X \text{ is not uniformly nonsquare}.
\]
\end{theorem}

\begin{proof}
Assume first that \(J_\perp(X)=2\).
Let \(\varepsilon>0\). Then there exist \(x,y\in S_X\) with \(x\perp_B y\) such that
\[
\min\{\|x+y\|,\|x-y\|\}>2-\varepsilon.
\]

Suppose \(1/2\leq\lambda<1\). Then
\[
\lambda x+(1-\lambda)y
=
\lambda(x+y)-(\lambda-(1-\lambda))y.
\]
Therefore
\[
\|\lambda x+(1-\lambda)y\|
\geq
\lambda\|x+y\|-(2\lambda-1)\|y\|
>
\lambda(2-\varepsilon)-(2\lambda-1)
=
1-\lambda\varepsilon.
\]
Similarly,
\[
\|\lambda x-(1-\lambda)y\|
>
1-\lambda\varepsilon.
\]
Hence
\[
J_\lambda^\perp(X)\geq 1-\lambda\varepsilon.
\]
Since \(\varepsilon>0\) is arbitrary and \(J_\lambda^\perp(X)\leq1\), we obtain
\[
J_\lambda^\perp(X)=1.
\]

The case \(0<\lambda\leq1/2\) is similar. In this case,
\[
\lambda x+(1-\lambda)y
=
(1-\lambda)(x+y)-((1-\lambda)-\lambda)x,
\]
so
\[
\|\lambda x+(1-\lambda)y\|
>
(1-\lambda)(2-\varepsilon)-(1-2\lambda)
=
1-(1-\lambda)\varepsilon.
\]
The same estimate holds for \(\lambda x-(1-\lambda)y\). Hence again
\[
J_\lambda^\perp(X)=1.
\]

Conversely, suppose \(J_\lambda^\perp(X)=1\) for some \(0<\lambda<1\).
By the Proposition \ref{p1},
\[
1
=
J_\lambda^\perp(X)
\leq
|2\lambda-1|+\min\{\lambda,1-\lambda\}J_\perp(X).
\]
Since
\[
1-|2\lambda-1|=2\min\{\lambda,1-\lambda\},
\]
we get
\[
2\min\{\lambda,1-\lambda\}
\leq
\min\{\lambda,1-\lambda\}J_\perp(X).
\]
Because \(0<\lambda<1\), the minimum is positive, and so
\[
J_\perp(X)\geq2.
\]
But always \(J_\perp(X)\leq2\), hence \(J_\perp(X)=2\).
\end{proof}

From the above theorem, we obtain the following two corollaries.

\begin{corollary}  Let $0<\lambda, \mu<1$. Then
	$$
	J_\lambda^{\perp}(X)=1 \quad \Longleftrightarrow \quad J_\mu^{\perp}(X)=1 .
	$$
	
\end{corollary}

\begin{corollary}
The following are equivalent:
\begin{enumerate}
\item \(X\) is uniformly nonsquare;
\item \(J_\lambda^\perp(X)<1\) for every \(\lambda\in(0,1)\);
\item \(J_\lambda^\perp(X)<1\) for some \(\lambda\in(0,1)\).
\end{enumerate}
\end{corollary}

\begin{remark} It is known that any uniformly non-square Banach space has the fixed point property \cite{15}. Thus, it can be known from the above corollary that when 
$
J_\lambda^{\perp}(X)<1,
$ then a Banach space $X$ has the fixed point property.
\end{remark}

\begin{proposition}
If \(X\) and \(Y\) are linearly isometric Banach spaces, then
\[
J_\lambda^\perp(X)=J_\lambda^\perp(Y)
\]
for every \(\lambda\in[0,1]\).
\end{proposition}

\begin{proof}
Let \(T:X\to Y\) be a surjective linear isometry. Then \(T(S_X)=S_Y\).

If \(x\perp_B y\) in \(X\), then for every \(t\in\mathbb R\),
\[
\|Tx+tTy\|
=
\|T(x+ty)\|
=
\|x+ty\|
\geq
\|x\|
=
\|Tx\|.
\]
Hence \(Tx\perp_B Ty\) in \(Y\). Similarly, applying \(T^{-1}\), every
Birkhoff--James orthogonal pair in \(Y\) comes from one in \(X\).

Moreover,
\[
\|\lambda Tx\pm(1-\lambda)Ty\|
=
\|T(\lambda x\pm(1-\lambda)y)\|
=
\|\lambda x\pm(1-\lambda)y\|.
\]
Thus the two suprema coincide.
\end{proof}

\begin{proposition}Let $X$ be strictly convex and let $0<\lambda<1$. If $x, y \in S_{X}$ and $x \perp_{B} y$, then

$$
m_{\lambda}(x, y)<1
$$
\end{proposition}
\begin{proof} For every $x, y \in S_{X}$,

$$
\|\lambda x+(1-\lambda) y\| \leq 1, \quad\|\lambda x-(1-\lambda) y\| \leq 1
$$

Suppose, for contradiction, that $m_{\lambda}(x, y)=1$. Then both norms must equal 1 :

$$
\|\lambda x+(1-\lambda) y\|=1, \quad\|\lambda x-(1-\lambda) y\|=1
$$

Since $0<\lambda<1$ and $X$ is strictly convex, equality in

$$
\|\lambda x+(1-\lambda) y\| \leq \lambda\|x\|+(1-\lambda)\|y\|=1
$$

implies $x=y$. But then $x \perp_{B} x$, this contradiction proves the result.
\end{proof}
\begin{corollary} If $X$ is finite-dimensional and strictly convex, then for every $\lambda \in(0,1)$,

$$
J_{\lambda}^{\perp}(X)<1.
$$
\end{corollary}
\begin{proof}By finite-dimensional attainment, there exist $x_{0}, y_{0} \in S_{X}$ with $x_{0} \perp_{B} y_{0}$ such that

$$
J_{\lambda}^{\perp}(X)=m_{\lambda}\left(x_{0}, y_{0}\right).
$$

By the preceding theorem,

$$
m_{\lambda}\left(x_{0}, y_{0}\right)<1.
$$

Thus

$$
J_{\lambda}^{\perp}(X)<1.
$$
\end{proof}

\begin{proposition} Let $X$ be finite-dimensional and let $0<\lambda<$ 1. If

$$
J_{\lambda}^{\perp}(X)=1
$$

then $S_{X}$ contains a nontrivial line segment. In particular, $X$ is not strictly convex.
\end{proposition} 

\begin{proof} By finite-dimensional attainment, there exist $x_{0}, y_{0} \in S_{X}$ with $x_{0} \perp_{B} y_{0}$ such that

$$
m_{\lambda}\left(x_{0}, y_{0}\right)=1
$$

Since each of the two norms is at most 1 , we must have

$$
\left\|\lambda x_{0}+(1-\lambda) y_{0}\right\|=1, \quad\left\|\lambda x_{0}-(1-\lambda) y_{0}\right\|=1
$$

The first equality says that an interior convex combination of $x_{0}$ and $y_{0}$ lies on the unit sphere. By convexity of the unit ball, the whole segment

$$
\left[x_{0}, y_{0}\right]=\left\{t x_{0}+(1-t) y_{0}: 0 \leq t \leq 1\right\}
$$

is contained in $S_{X}$. Since $x_{0} \neq y_{0}$, this is a nontrivial line segment. Hence $X$ is not strictly convex.
\end{proof}

The modulus of convexity has multiple equivalent formulations (see \cite{22}). We present one such equivalent definition below, and the subsequent theorem justifies why this particular form is adopted.

 For $\epsilon \in[0,2],$ the modulus of convexity is defined :

$$
\delta_X(\epsilon)=\inf \left\{1-\frac{\|x+y\|}{2}: x, y \in S_X ;\|x-y\| \geq \epsilon\right\} .
$$

\begin{theorem} For every Banach space $X$ and every $\lambda \in[0,1]$,

$$
J_{\lambda}^{\perp}(X) \leq 1-2 \min \{\lambda, 1-\lambda\} \delta_{X}(1)
$$
\end{theorem} 
\begin{proof}Let $x, y \in S_{X}$ with $x \perp_{B} y$, then

$$
\|x-y\| \geq 1, \quad\|x+y\| \geq 1
$$

Assume first that $0 \leq \lambda \leq 1 / 2$. Then

$$
\lambda x+(1-\lambda) y=2 \lambda \frac{x+y}{2}+(1-2 \lambda) y
$$

Hence

$$
\|\lambda x+(1-\lambda) y\| \leq 2 \lambda\left\|\frac{x+y}{2}\right\|+1-2 \lambda
$$

Since $\|x-y\| \geq 1$,

$$
\left\|\frac{x+y}{2}\right\| \leq 1-\delta_{X}(1)
$$

Thus

$$
\|\lambda x+(1-\lambda) y\| \leq 1-2 \lambda \delta_{X}(1)
$$

Similarly,

$$
\lambda x-(1-\lambda) y=2 \lambda \frac{x-y}{2}-(1-2 \lambda) y .
$$

Since $\|x+y\| \geq 1$, we also get

$$
\|\lambda x-(1-\lambda) y\| \leq 1-2 \lambda \delta_{X}(1)
$$

Therefore

$$
m_{\lambda}(x, y) \leq 1-2 \lambda \delta_{X}(1)
$$

If $1 / 2 \leq \lambda \leq 1$, then

$$
\lambda x+(1-\lambda) y=(2 \lambda-1) x+2(1-\lambda) \frac{x+y}{2}
$$

and similarly for the minus sign. The same argument yields

$$
m_{\lambda}(x, y) \leq 1-2(1-\lambda) \delta_{X}(1).
$$

Combining the two cases gives

$$
m_{\lambda}(x, y) \leq 1-2 \min \{\lambda, 1-\lambda\} \delta_{X}(1).
$$

Taking the supremum over all admissible pairs proves the theorem.\end{proof}

\begin{corollary} If $X$ is uniformly convex, then for every $\lambda \in(0,1)$,

$$
J_{\lambda}^{\perp}(X)<1.
$$
\end{corollary}
\begin{proof} If $X$ is uniformly convex, then
$
\delta_{X}(1)>0 (see \cite{23}).
$

Therefore, by the preceding theorem,

$$
J_{\lambda}^{\perp}(X) \leq 1-2 \min \{\lambda, 1-\lambda\} \delta_{X}(1)<1
$$

for every $\lambda \in(0,1)$.
\end{proof}

\begin{proposition} Let $0<\lambda<1$. For a finite-dimensional Banach space $X$,
$
J_{\lambda}^{\perp}(X)=1
$
if and only if there exist $x, y \in S_{X}$ with $x \perp_{B} y$ such that

$$
\lambda x+(1-\lambda) y \in S_{X} \quad \text { and } \quad \lambda x-(1-\lambda) y \in S_{X} .
$$
\end{proposition}
\begin{proof} Assume first that $J_{\lambda}^{\perp}(X)=1$. By finite-dimensional attainment, there exist $x, y \in S_{X}$ with $x \perp_{B} y$ such that
$
m_{\lambda}(x, y)=1 \text {. }
$

Since both terms in the minimum are at most 1 , both must equal 1 . Hence

$$
\lambda x+(1-\lambda) y \in S_{X}, \quad \lambda x-(1-\lambda) y \in S_{X} .
$$

Conversely, if such $x, y$ exist, then
$
m_{\lambda}(x, y)=1 \text {. }
$

Since $J_{\lambda}^{\perp}(X)$ is the supremum of $m_{\lambda}$ over all admissible pairs and is always at most 1 , we get
$
J_{\lambda}^{\perp}(X)=1 .
$
\end{proof}
\begin{corollary}Let $X$ be finite-dimensional and $0<\lambda<1$. If

$$
J_{\lambda}^{\perp}(X)=1
$$

then there exist $x, y \in S_{X}$ with $x \perp_{B} y$ such that both segments
$[x, y]$ and $ [x,-y]$ are contained in $S_{X}$.\end{corollary}

\begin{proof} By the preceding theorem, there exist $x, y \in S_{X}$ such that

$$
\|\lambda x+(1-\lambda) y\|=1, \quad\|\lambda x-(1-\lambda) y\|=1 .
$$

Since $\lambda x+(1-\lambda) y$ is an interior convex combination of $x$ and $y$, convexity of $B_{X}$ implies that the whole segment $[x, y]$ lies on $S_{X}$. Similarly, $\lambda x+(1-\lambda)(-y)$ is an interior convex combination of $x$ and $-y$, and hence
$
[x,-y] \subset S_{X}.
$
\end{proof}

\begin{proposition}
	\label{thm:smoothness}
	Let $X$ be a real Banach space and $0<\lambda\le 1$. Then
	\[
	J_\lambda^\perp(X)
	\le
	\min\left\{
	1,\,
	\lambda\left(1+\rho_X\!\left(\frac{1-\lambda}{\lambda}\right)\right)
	\right\}.
	\]
	In particular, for $\lambda=\frac12$,
	\[
	J_{1/2}^\perp(X)\le \frac{1+\rho_X(1)}{2}.
	\]
\end{proposition}

\begin{proof}
	Let $x,y\in S_X$ with $x\perp_B y$, then by the definition of $\rho_X$,
	\[
	\frac{\|x+\frac{1-\lambda}{\lambda}y\|+\|x-\frac{1-\lambda}{\lambda}y\|}{2}
	\le
	1+\rho_X\!\left(\frac{\lambda}{1-\lambda}\right).
	\]
	Multiplying by $\lambda$ yields
	\[
	\frac{\|\lambda x+(1-\lambda) y\|+\|\lambda x-(1-\lambda)y\|}{2}
	\le
	\lambda\left(1+\rho_X\!\left(\frac{1-\lambda}{\lambda}\right)\right).
	\]
	Since the minimum of two nonnegative numbers is bounded above by their average,
	\[
	\min\{\|\lambda x+(1-\lambda) y\|,\|\lambda x-(1-\lambda) y\|\}
	\le
	\lambda \left(1+\rho_X\!\left(\frac{(1-\lambda)}{\lambda}\right)\right).
	\]
	Taking the supremum over all admissible pairs gives
	\[
	J_\lambda^\perp(X)
	\le
	\lambda\left(1+\rho_X\!\left(\frac{1-\lambda}{\lambda}\right)\right).
	\]
	On the other hand, the triangle inequality gives
	\[
	\|\lambda x\pm(1-\lambda)y\|
	\le
	\lambda+(1-\lambda)=1,
	\]
	hence $J_\lambda^\perp(X)\le 1$. Combining the two estimates proves the theorem.
\end{proof}

\begin{proposition}
	\label{thm:CNJ}
	For every real Banach space $X$ and every $\lambda\in[0,1]$,
	\[
	J_\lambda^\perp(X)
	\le
	\min\left\{
	1,\,
	\sqrt{C_{\mathrm{NJ}}(X)\big(\lambda^2+(1-\lambda)^2\big)}
	\right\}.
	\]
\end{proposition}

\begin{proof}
	Let $x,y\in S_X$ with $x\perp_B y$, and put
	\[
	u=\lambda x,\qquad v=(1-\lambda)y.
	\]
	By the definition of $C_{\mathrm{NJ}}(X)$,
	\[
	\frac{\|u+v\|^2+\|u-v\|^2}{2(\|u\|^2+\|v\|^2)}
	\le
	C_{\mathrm{NJ}}(X).
	\]
	Therefore
	\[
	\frac{\|\lambda x+(1-\lambda)y\|^2
		+
		\|\lambda x-(1-\lambda)y\|^2}{2}
	\le
	C_{\mathrm{NJ}}(X)\big(\lambda^2+(1-\lambda)^2\big).
	\]
	The smaller of two nonnegative numbers is bounded by the square root of their
	quadratic mean, so
	\[
	\min\{\|\lambda x+(1-\lambda)y\|,
	\|\lambda x-(1-\lambda)y\|\}
	\le
	\sqrt{C_{\mathrm{NJ}}(X)\big(\lambda^2+(1-\lambda)^2\big)}.
	\]
	Taking the supremum gives the asserted estimate. The triangle inequality again
	gives the additional bound $J_\lambda^\perp(X)\le 1$.
\end{proof}

\section*{Acknowledgments}

Thanks to all the members of the Functional Analysis Research Team at the School of Mathematics and Statistics, Anqing Normal University, for their valuable discussions and corrections regarding the
challenges and errors encountered in this article. 

We also thank Dr. Yuankang Fu for his valuable comments on this paper.
\raggedright

\end{document}